\documentclass[letterpaper, 10 pt, conference]{ieeeconf}  %

\IEEEoverridecommandlockouts                              %
\usepackage[utf8]{inputenc}
\usepackage[T1]{fontenc}

\usepackage{graphics} %
\usepackage{amsmath} %
\usepackage{amsfonts}
\usepackage{amssymb}  %
\usepackage{mathrsfs}
\usepackage[binary-units]{siunitx}
\usepackage{pgfplots}
\usepackage[commentmarkup=uwave]{changes}

\makeatletter
\let\NAT@parse\undefined
\makeatother

\let\ieeebibliography\thebibliography

\usepackage[numbers,sort&compress]{natbib}
\usepackage[colorlinks]{hyperref}

\renewcommand\thebibliography[1]{\ieeebibliography{#1}}

\usepackage[IEEE]{altthm-styles}

\definechangesauthor[name={Benoît}, color=blue]{BL}
\definechangesauthor[name={Torbjørn}, color=red]{TC}

\renewcommand{\added}[2][]{#2}

\newcommand{\Def}[1]{\overset{\text{\tiny def}}{#1}}

\title{\LARGE \bf
Sequential sum-of-squares programming for analysis of nonlinear systems$^\star$
}

\author{Torbjørn Cunis$^{1}$ and Benoît Legat$^{2}$%
\thanks{$^\star$Supported by the authors' institutions.}%
\thanks{$^{1}$T. Cunis is with the Institute of Flight Mechanics and Control,
        University of Stuttgart, 70569 Stuttgart, Germany.
        {\tt\small tcunis@ifr.uni-stuttgart.de}}%
\thanks{$^{2}$B. Legat is with the ESAT Deparment, KULeuven, Leuven, Belgium.
        {\tt\small benoit.legat@esat.kuleuven.be}}%
}

\begin{document}

\maketitle
\thispagestyle{empty}
\pagestyle{empty}

\begin{abstract}
Numerous interesting properties in nonlinear systems analysis can be written as polynomial optimization problems with nonconvex sum-of-squares problems. To solve those problems efficiently, we propose a sequential approach of local linearizations leading to tractable, convex sum-of-squares problems. Local convergence is proven under the assumption of strong regularity and the new approach is applied to estimate the region of attraction of a polynomial aircraft model.
\end{abstract}

\section{Introduction}
Polynomials that can be written as a sum of squares are a strict subset of the nonnegative polynomials. While determining whether a given polynomial does not assume negative values is computationally hard, Parillo~\cite{parillo2003} showed in his seminal paper that convex optimization over sum-of-squares polynomials can be reduced to semidefinite programming. His works, as well as the development of the dual approach via moments by Lasserre~\cite{Lasserre2001}, and the advent of efficient algorithms for semidefinite problems laid the foundation for numerical analysis of nonlinear systems with polynomial dynamics that is today known as sum-of-squares programming.

Applications of convex sum-of-squares programming include stability verification for hybrid systems \cite{Prajna2003, ahmadiParrilo2008, papachristodoulouPrajna2009, Ichihara2012, Holicki2019}, optimization algorithms and optimization-based control \cite{summersKunzEtAl2013, Tan2021, Korda2017}, control synthesis \cite{jarvisEtAl2003, ebenbauerAllgoewer2006, Oishi2013, Vatani2015}, and many more. 
As these approaches often make use of Lyapunov-type functions and dissipativity inequalities, many sum-of-squares constraints for polynomial dynamics can be viewed as the natural extension of linear matrix inequalities for linear systems~\cite{Boyd1994}. However, unlike in the linear case, most properties of nonlinear systems such as asymptotic stability, invariance, or controllability often are valid on a region of the system's state-space only. The problem of determining the region of attraction \cite{Genesio1989}, for example, thus consists of finding a Lyapunov candidate $V$ and a region $X$ (often a sublevel set) as well as certifying that $V$ decays strictly on $X$. 
If $V$ and the describing function of $X$ are polynomial decision variables, estimating the region of attraction is a nonconvex, nonlinear sum-of-squares problem by the Positivstellensatz of the reals \cite{stengle1974}.

Despite nonlinear sum-of-squares problems being computationally hard, local analysis of stability and other properties of polynomial dynamics with sum-of-squares programming has been extensively studied \cite{Cotorruelo2020a, khodadadiEtAl2014, Cunis2020c, zhengEtAl2018, topcuEtAl2010, topcuEtAl2008, Newton2022, Cotorruelo2021, Iannelli2019, Yin2021b, Cunis2020a}. Here, the (mostly bilinear) nonconvex constraints have been mitigated by bisections \cite{seilerBalas2010}, coordinate descent \cite{majumdarEtAl2013}, and combinations of both. Yet, except for quasiconvex problems, convergence is not guaranteed (see remarks in \cite{chakrabortyEtAl2011}). Given that the underlying semidefinite problems scale notoriously with the polynomial degree, it is desirable to limit the number of convex evaluations.

In this paper, we take inspiration from sequential convex programming \cite{Freund2007, Mao2018, Li2019, Dinh2010, Doelman2017} and study a sequential approach for nonlinear conic problems which we combine with a line search using a merit function from Powell~\cite{Powell1978}. The nonlinear problem is linearized around a solution candidate in order to obtain an affine conic problem. For the sum-of-squares cone, sum-of-squares toolboxes such as {\sc spot}, {\sc sosopt}, or {\sc sostools} are readily available to solve the local problems via reduction to a semidefinite program; and more recently, direct implementations of the sum-of-squares cone have been proposed \cite{Papp2019,legat2022low}.
{Similar to the affine case, nonlinear sum-of-squares problems could directly be reduced to a nonlinear semidefinite program; yet the authors are only aware of the toolbox {\sc SumOfSquares.jl}~\cite{weisser2019polynomial} for that purpose, which is limited to quadratic expressions.}
Moreover, the semidefinite representation of a sum-of-squares polynomial is nonunique (see comments in \cite[Section~3.2]{parillo2003}); yet uniqueness of the solution is usually assumed for convergence. %

We prove local convergence of the sequence of convex problems using a result from variational analysis \cite{Dontchev2021} that builds upon the implicit function theorem for strongly regular generalized equations by Robinson \cite{Robinson1980}.
As this result is stated for (possibly infinite dimensional) Banach spaces,
our analysis works in the general setting of nonlinear conic programs with convex cones embedded in Banach spaces.
Since the vector space of polynomials is not complete, we limit ourselves to optimization problems with fixed polynomial degree but our sequential algorithm can be applied to other cones as well. We further investigate the line search based on the dual theory of affine sum-of-squares optimization. Numerical results for practical engineering problems demonstrate that sequential sum-of-squares programming significantly reduces the number of convex problems to be solved and thus the computation time compared to previous, iterative approaches.

Our proof generalizes \cite{Dinh2010} in two aspects.
First, we consider convex cones in arbitrary Banach spaces rather than embedded in $\mathbb R^n$.
Second, we show that the convergence still holds if a line search is used to improve convergence speed. %
Moreover, by use of variational analysis, our paper provides a simpler proof %
while obtainig a tighter convergence rate. %

The remainder of the paper is organized as follows: Section~\ref{sec:preliminaries} introduces the tools from variational analysis and Section~\ref{sec:sumofsquares} motivates and states the problem of nonlinear sum-of-squares optimization. The sequential programming approach is detailed in Section~\ref{sec:sequentialsos} and local convergence is proven in Section~\ref{sec:analysis}. In Section~\ref{sec:examples}, the sequential approach is applied to problems from nonlinear system analysis. %

\paragraph*{Notation}
$\mathbb N$ (resp., $\mathbb N_0$) and $\mathbb R$ denote the natural numbers excluding (resp., including) zero and the reals, respectively. For some $m \in \mathbb N$, the set of symmetric (resp., positive semidefinite) matrices in $\mathbb R^{m \times m}$ is $\mathbb S_m$ (resp., $\mathbb S_m^+$).

\section{Preliminaries}
\label{sec:preliminaries}
Let $X$, $Y$, and $P$ be Banach spaces.
The dual space $X^*$ is set of linear operators $l: X \to \mathbb R$ is $X^*$ %
with evaluation $\langle \cdot, \cdot \rangle: X^* \times X \to \mathbb R$.
{Moreover, the adjunct of a linear mapping $A: X \to Y$ is the linear mapping $A^*: Y^* \to X^*$ satisfying $\langle l, A(\xi) \rangle = \langle A^*(l), \xi \rangle$ for all $\xi \in X$ and $l \in Y^*$.}

\subsection{Normals \& Gradients}
A convex cone is a set $C \subset X$ satisfying $r_1 \xi_1 + r_2 \xi_2 \in C$ for all $\xi_1, \xi_2 \in C$ and $r_1, r_2 \in \mathbb R_{\geq 0}$.
The dual cone of $C$ is defined as %
\begin{align*}
    C^* = \{ v \in X^* \, | \, \text{$\langle v, \xi \rangle \geq 0$ for all $\xi \in X$} \}
\end{align*}
and the dual of $C^*$ is isometric to $C$.
Moreover, for a convex set $\Omega \subset X$, the normal cone mapping $N_\Omega: X \rightrightarrows X^*$ is given by %
\begin{align*}
    N_\Omega(\xi_0) = \{ w \in X^* \, | \, \text{$\langle w, \xi - \xi_0 \rangle \leq 0$ for all $\xi \in \Omega$} \}
\end{align*}
if $\xi_0 \in \Omega$, and $N_\Omega(\xi_0) = \varnothing$ otherwise.

\begin{definition}
    The {\em (Fréchet) derivative} of a nonlinear function $g: X \to Y$ at $\xi_0 \in X$ is a linear mapping $\nabla g(\xi_0): X \to Y$ satisfying %
    \begin{align*}
        \lim_{\xi \to \xi_0} \frac{g(\xi_0) + \nabla g(\xi_0)(\xi - \xi_0) - g(\xi) }{\| \xi_0 - \xi \|} = 0 
    \end{align*}
    and $g$ is {\em (Fréchet) differentiable} if and only if $\nabla g(\xi_0)$ exists for all $\xi_0 \in X$.
\end{definition}

\subsection{Continuity \& Regularity}
Let $h: X \times P \to Y$ be a function and $H: X \rightrightarrows Y$ be a set-valued mapping.
$h$ is said to be {\em Lipschitz continuous} with respect to $\xi$ around $(\xi_0, \pi_0) \in \operatorname{int}\operatorname{dom} h$ if and only if 
\begin{align*}
    \limsup_{\substack{\xi, \xi' \to \xi_0, \xi \neq \xi' \\ \pi \to \pi_0}} \frac{\| h(\xi', \pi) - \psi(\xi, \pi) \|}{\| \xi' - \xi \|} = \kappa
\end{align*}
with constant $\kappa < \infty$. 
Moreover, $H$ is said to have a {\em single-valued localization} $h: X' \to Y'$ around $\xi_0 \in X$ for $\upsilon_0 \in Y$ if and only if $X' \subset X$ and $X' \subset X$ are neighbourhoods of $\xi_0$ and $\upsilon_0$, respectively, $h(\xi_0) = \upsilon_0$, and $H(\xi) \cap Y' = \{h(\xi)\}$ for all $\xi \in X'$. 
If the inverse $H^{-1}: \upsilon \mapsto \{ \xi \in X \, | \, \upsilon \in H(\xi) \}$ has a single-valued localization %
around $\upsilon_0$ for $\xi_0$ that is Lipschitz continuous around $\upsilon_0$ with constant $\gamma$, then $H$ is called {\em strongly regular} at $\xi_0$ for $\upsilon_0$ with constant $\gamma$.

\begin{theorem}[Theorem~8.8 of \cite{Dontchev2021}]
    \label{thm:implicit-vi}
    Take $\psi: X \times P \to Y$ and $N: X \rightrightarrows Y$;
    suppose $\psi$ is Lipschitz continuous with respect to $\pi$ around $(\xi_0, \pi_0)$ with constant $\kappa$ and $0 \in \psi(\xi_0, \pi_0) + N(\xi_0)$; if there exists $\psi_0: X \to Y$ such that $\psi_0(\xi) = \psi(\xi, \pi)$ around $\xi_0$ if $\pi \to \pi_0$ and $\psi_0 + N$ is strongly regular at $\xi_0$ for $0$ with constant $\gamma$, then the mapping
    \begin{align*}
        H: \pi \mapsto \{ \xi \in X \, | \, \psi(\xi, \pi) + N(\xi) \ni 0 \}
    \end{align*}
    has a Lipschitz continuous, single-valued localization around $\pi_0$ for $\xi_0$ with constant $\gamma \kappa$.
    \noqed
\end{theorem}

\section{Sum-of-squares Optimization}
\label{sec:sumofsquares}
Let $x = (x_1, \ldots, x_n)$ be a tuple of free variables and $\alpha = (\alpha_1, \ldots, \alpha_n) \in \mathbb N_0^n$ a multi-index; a polynomial $\pi$ in $x$ up to degree $d$ is a linear combination
\begin{align*}
    \pi = \sum_{\| \alpha \|_1 \leq d} c_\alpha x^\alpha
\end{align*}
where $x^\alpha = x_1^{\alpha_1} \cdots x_n^{\alpha_n}$ and $\| \alpha \|_1 = \sum_i \alpha_i$. The set $\mathbb R_d[x]$ of polynomials in $x$ with real coefficients $c_\alpha \in \mathbb R$ up to degree $d$ forms a vector space with norm $\| \cdot \|$. %

\begin{definition}
    A polynomial $\pi \in \mathbb R_d[x]$ is {a} {\em sum-of-squares} {polynomial} ($\pi \in \Sigma_d[x]$) if and only if there exist {$m \in \mathbb N$ and} $\pi_1, \ldots, \pi_m \in \mathbb R_d[x]$ such that $\pi = \sum_{i=1}^m (\pi_i)^2$.
\end{definition}

It is easy to see that $\Sigma_d[x]$ forms a convex cone in $\mathbb R_d[x]$. Moreover, its dual cone %
$\Sigma_d[x]^*$
is isometric to the cone of sum-of-squares polynomials \cite{Lasserre2001}.

{To avoid confusion with the Fréchet derivative $\nabla$ (with respect to the space of polynomials), we are going to use $\partial_x: \mathbb R_d[x] \to \mathbb R_d[x]^{1 \times n}$ for the differentiation operator with respect to the free variables $x$.}

A convex sum-of-squares optimization problem
can be reduced to a semidefinite program. This is the fundamental result of \cite[Theorem~3.3]{parillo2003}, which reads as follows; denote by $s_d \in \mathbb N_0$ the number of {monomials up to degree $d$} of a polynomial in $x$. %
A polynomial $\pi \in \mathbb R_{2d}[x]$ is sum-of-squares
    if and only if there exists a matrix $Q \in \mathbb S_{s_d}^+$ satisfying $\pi = \zeta^\top Q \zeta$, where $\zeta \in \mathbb R_d[x]^{s_d}$ is the vector of monomials up to degree $d$.
The solution for $Q$ is usually not unique, leading to the {\em implicit} (or kernel) and {\em explicit} (or image) relaxations of an affine sum-of-squares problem.

\subsection{Motivation}
Consider a continuous-time dynamic system defined by the differential equation
\begin{align}
    \dot x = \phi(x)
\end{align}
where $x \in \mathbb R^n$ denotes the state vector and $\phi: \mathbb R^n \to \mathbb R^n$ is a polynomial function satisfying $\phi(0) = 0$. Many system-theoretic properties on a domain $\mathcal D \subset \mathbb R^n$ can be written as {polynomial} dissipativity inequality of the form %
\begin{align}
    \label{eq:dissipativity}
    \forall x \in \mathbb R^n, \; x \in \mathcal D \Longrightarrow \partial_x V(x) \phi(x) \leq S(x)
\end{align}
where $V \in \mathbb R[x]$ and $S \in \mathbb R[x]$ are called storage function and supply rate, respectively. The set $\mathcal D$ usually depends on $V$; for example, if $\mathcal D$ is a sublevel set of $V$ and $V$ and $S$ are positive definite {polynomials}, then Eq.~\eqref{eq:dissipativity} is LaSalle's condition for asymptotic stability \cite[Theorem~2]{laSalle1960}. %

In the sum-of-squares literature, the dissipativity condition is rewritten using the so-called generalized S-procedure \cite{seilerBalas2010}. Suppose $\mathcal D = \{ x \in \mathbb R^n \, | \, \ell_V(x) \geq 0 \}$ for some $\ell_V \in \mathbb R_d[x]$, then if there exists $\varsigma \in \Sigma_d[x]$ such that %
\begin{align}
    \label{eq:dissipativity-sos}
    (S - \partial_x V \phi) - \varsigma \ell_V \in \Sigma_{d'}[x]
\end{align}
then $V$ and $S$ satisfy \eqref{eq:dissipativity}. However, since $\ell_V$ depends on $V$, the sum-of-squares constraint \eqref{eq:dissipativity-sos} is nonlinear. Previous approaches to optimize over nonlinear sum-of-squares constraint relied on solving for one variable at a time while keeping the remaining variables fixed {(see, e.g., \cite{seilerBalas2010, majumdarEtAl2013, chakrabortyEtAl2011})}.

\subsection{Nonlinear optimization problem}
In general, a polynomial optimization problem with nonconvex sum-of-squares constraints takes the form of a nonlinear optimization
\begin{align}
    \label{eq:sos-nonlinear}
    \min_{\xi \in X} \; \langle f, \xi \rangle \quad \text{s.t. $g(\xi) \in D$ and $\xi \in C$}
\end{align}
where $X$ and $Y$ are Banach spaces, $f \in X^*$ is a linear cost, $g: X \to Y$ is a differentiable constraint mapping, and $C \subset X$ and $D \subset Y$ are convex cones. In the case of sum-of-squares, $X$ and $Y$ correspond to spaces of polynomial up to a finite degree with sum-of-squares cones $C$ and $D$ and $g$ takes polynomial values.

Define the Lagrangian as $L(\xi,l) = \langle f, \xi \rangle - \langle l, g(\xi) \rangle$, where $l \in Y^*$ is a Lagrange multiplier; %
the Karush-Kuhn-Tucker (KKT) conditions for \eqref{eq:sos-nonlinear} at $\xi_0 \in X$ are
\begin{gather*}
    f - \nabla g(\xi_0)^* l - s = 0 \\
    \langle s, \xi_0 \rangle = 0, \quad \langle l, g(\xi_0) \rangle = 0 \\
    \xi_0 \in C, \quad g(\xi_0) \in D
\end{gather*}
where $s \in C^*$ and $l \in D^*$ are dual variables associated with the cone constraints.

Define $\vartheta = (\xi, l)$ and $\mathcal T = X \times Y^*$. %
With a small abuse of notation, we identify $g(\xi)$ as an element of $D^{**}$ using the canonical isomorphism between $D$ and $D^{**}$.
Then the KKT conditions are equivalent to $(-s, -g(\xi))$ belonging to the normal cone of $C \times D^*$ at $(\xi_0, l_0) \in \mathcal T$. 
The KKT conditions %
can thus be written as generalized equation
\begin{align}
    \label{eq:variational-inequality}
    \varphi(\vartheta) + N(\vartheta) \ni 0
\end{align}
where $\varphi: \mathcal T \to \mathcal T^*$ and $N: \mathcal T \rightrightarrows \mathcal T^*$ are defined as
\begin{gather*}
    \varphi: \vartheta \mapsto (f - \nabla g(\xi)^* l, \; g(\xi))
\end{gather*}
and
\begin{gather*}
    N: \vartheta \mapsto \{ (v, \zeta) \in \mathcal T^* \, | \,  v \in N_{C}(\xi), \zeta \in N_{D^*}(l) \}
\end{gather*}
respectively. We denote the solutions to \eqref{eq:variational-inequality} by $\Theta \subset \mathcal T$. 

\begin{assumption}
    \label{ass:kkt-points}
    The set $\Theta$ is nonempty.
\end{assumption}

{Under a suitable constraint qualification, existence of a KKT point is necessary for an optimal solution of \eqref{eq:sos-nonlinear}.}

\section{Sequential Programming}
\label{sec:sequentialsos}
We propose to approach the nonlinear problem \eqref{eq:sos-nonlinear} with a sequence of local, convex problems. To that extent, let $\xi^k \in X$ with $k \in \mathbb N_0$ be the solution of the $k$-th iteration and $l^k \in Y^*$ be the associated Lagrange multiplier. Pick tolerances $\epsilon_k, \epsilon_k^* > 0$ for the primal and dual solutions as well as a small weight $\eta > 0$. Our next instance $(\xi^{k+1}, l^{k+1})$ is subject to the steps:
\begin{enumerate}
    \item Solve the convex problem at $\xi^k$,
        \begin{subequations}
            \label{eq:sos-linearized}
        \begin{align}
            &\min_{\xi \in X, \; \varsigma \in Y} \langle f, \xi \rangle \\
            &\quad \text{s.t.} \quad g(\xi^k) + \nabla g(\xi^k) (\xi - \xi^k) = \varsigma \\
            &\quad \text{and} \quad \xi \in C, \quad \varsigma \in D
        \end{align}
        \end{subequations}
        and denote the optimal solution as $\xi_+$ and the associated Lagrange multiplier as $l_+$.
    \item Solve the line search
        \begin{align}
            \label{eq:linesearch}
            \min_{r \in \mathbb R} \psi(r) - \eta r \quad \text{s.t. $0 < r \leq 1$}
        \end{align}
        where $\psi: r \mapsto L(r \xi_+ + (1-r) \xi^k, l_+)$,
        and denote the optimal solution as $\hat r$.
    \item Set $\xi^{k+1} = \hat r \xi_+ + (1 - \hat r) \xi^k$ and $l^{k+1} = \hat r l_+ + (1 - \hat r) l^k$.
\end{enumerate}
We terminate the iteration if both $\| \xi^{k+1} - \xi^k \| \leq \epsilon_k$ and $\| l^{k+1} - l^k \|_* \leq \epsilon_k^*$, where $\| \ell \|_*$ denotes the operator norm of $\ell: Y \to \mathbb R$. Otherwise, we repeat the steps for $k+1$.

As linear problem, \eqref{eq:sos-linearized} has a dual problem at $\xi^k$, viz.
\begin{subequations}
    \label{eq:sos-linearized-dual}
\begin{align}
    &\max_{l \in Y^*, \; s \in X^*} \langle l, \gamma_k \rangle \\ %
    &\quad \text{s.t.} \quad f - \nabla g(\xi^k)^* l - s = 0 \\
    &\quad \text{and} \quad l \in D^*, \quad s \in C^*
\end{align}
\end{subequations}
where $\gamma_k \Def= \nabla g(\xi^k)\xi^k - g(\xi^k)$. 
In the following analysis, we will assume that the Lagrange multiplier $l_+$ is the optimal solution of \eqref{eq:sos-linearized-dual}, provided it exists. 
If $(\xi_0, l_0) \in \mathcal T$ satisfy the KKT conditions 
\begin{gather*}
    f - \nabla g(\xi^k)^* l_0 - s = 0 \\
    \langle s, \xi_0 \rangle = 0, \quad \langle l_0, \nabla g(\xi^k) \xi_0 - \gamma_k \rangle = 0 \\
    (\xi_0, l_0) \in C \times D^*, \quad \nabla g(\xi^k) \xi_0 - \gamma_k \in D
\end{gather*}
{for some $s \in C^*$,}
then $\xi_0$ and $l_0$ are optimal solutions for \eqref{eq:sos-linearized} and \eqref{eq:sos-linearized-dual}, respectively, and satisfy $\langle f, \xi_0 \rangle = \langle l_0, \gamma_k \rangle$.

\section{Theoretical Analysis}
\label{sec:analysis}
We are going to prove local convergence of the \added[id=BL]{sequential algorithm} using a parametrized version of the generalized equation \eqref{eq:variational-inequality}; define
\begin{align*}
    \hat \varphi: (\vartheta, \xi^k) \mapsto (f - \nabla g(\xi^k)^* l, \; \nabla g(\xi^k) \xi - \gamma_k)
\end{align*}
then $\vartheta_0 = (\xi_0, l_0) \in \Theta$ if and only if it solves %
\begin{align}
    \label{eq:vi-linearized}
    \hat L(\vartheta, \xi^k) \Def= \hat \varphi(\vartheta, \xi^k) + N(\vartheta) \ni 0
\end{align}
at $\xi^k = \xi_0$. In other words, the generalized equation \eqref{eq:vi-linearized} can be understood as linearization of \eqref{eq:variational-inequality} around $\xi^k$.
The set of KKT points of \eqref{eq:sos-linearized} at $\xi^k \in X$ is given by 
\begin{align*}
    H(\xi^k) = \{ \vartheta \in \mathcal T \, | \, \hat L(\vartheta, \xi^k) \ni 0 \}
\end{align*}
the {\em solution map} of \eqref{eq:vi-linearized}.

\pagebreak

The following result is a special case of \cite[Theorem~2F.1]{Dontchev2011} for solution mappings of monotone variational inequalities and proved here for completeness.

\begin{lemma}
    \label{lem:optimality}
    Let $\xi^k \in X$; 
    if $H(\xi^k)$ is nonempty, then $H(\xi^k)$ is a convex set containing $(\xi_+, l_+)$.
\end{lemma}
\begin{proof}
    Denote the set of optimal solutions to \eqref{eq:sos-linearized} and~\eqref{eq:sos-linearized-dual} by $S \subset X \times Y^*$.
    Assume that $H(\xi^k)$ is nonempty, take $(\xi_0, l_0) \in H(\xi^k)$ and $(\xi_+, l_+) \in S$.
    By sufficiency of the KKT conditions, $(\xi_0, l_0) \in S$ and $\langle f, \xi_0 \rangle - \langle l_0, \gamma_k \rangle = 0$. Since $\langle f, \xi_+ \rangle \leq \langle f, \xi_0 \rangle$ and $\langle l_+, \gamma_k \rangle \geq \langle l_0, \gamma_k \rangle$ by primal and dual optimality, 
    \begin{align*}
        0 &\geq \langle f, \xi_+ \rangle - \langle l_+, \gamma_k \rangle \\
        &= \langle \nabla g(\xi^k)^* l_+ + s, \xi_+ \rangle - \langle l_+, \gamma_k \rangle \\
        &= \langle l_+, \nabla g(\xi^k) \xi_+ - \gamma_k \rangle + \langle s, \xi_+ \rangle
    \end{align*}
    with $s \in C^*$.
    Since $\langle l_+, \nabla g(\xi^k) \xi_+ - \gamma_k \rangle \geq 0$ and $\langle s, \xi_+ \rangle \geq 0$, the inequalities are tight and $(\xi_+, l_+) \in H(\xi^k)$. Hence, $H(\xi^k) = S$, a convex set.
\end{proof}

Combining Assumption~\ref{ass:kkt-points} with Lemma~\ref{lem:optimality}, a KKT point $(\xi_0, l_0)$ of the nonlinear problem \eqref{eq:sos-nonlinear} is a %
candidate stationary condition of the sequential algorithm.
However, we have yet to prove that $(\xi_+, l_+)$ is the unique solution of the parametrized variational inequality \eqref{eq:vi-linearized} around $\xi_0$. %
To that extent,
we make the following, standing assumptions.

\begin{assumption}
    \label{ass:regularity}
    For all $(\xi_0, l_0) \in \Theta$, the mapping $L_0 = \hat L(\cdot, \xi_0)$ is strongly regular at $(\xi_0, l_0)$ for 0 with constant $\gamma$.
\end{assumption}

By definition, strong regularity of $L_0$ requires that $L_0^{-1}(\delta)$ has a single-valued localization around $\delta = 0$ for $\vartheta_0$, which is equivalent to a perturbed convex problem having unique solutions for small perturbations $\delta \in \mathcal T^*$ \cite{Dinh2010}. 

\begin{assumption}
    \label{ass:lipschitz}
    For all $(\xi_0, l_0) \in \Theta$, the gradient $\nabla g(\xi)$ is Lipschitz continuous around $\xi_0$
    and the mapping $\nabla g(\xi)^* l$ has the Lipschitz constant $\kappa$ with respect to $\xi$ around $(\xi_0, l_0)$.
\end{assumption}

If $g$ is twice differentiable at $\xi_0$, then $\nabla g(\xi)$ is Lipschitz continuous at $\xi_0$ and the constant $\kappa$ is determined by the norm of its second derivative.
We note the following implication of Assumption~\ref{ass:lipschitz}.

\begin{lemma}
    \label{lem:regularity}
    Let $\vartheta_0 = (\xi_0, l_0) \in \Theta$; the mapping $\hat \varphi(\vartheta, \xi^k)$ is Lipschitz continuous with respect to $\xi^k$ around $(\vartheta_0, \xi_0)$ with constant $\kappa$.
\end{lemma}
\begin{proof}
    By Assumption~\ref{ass:lipschitz}, %
    the first component of $\hat \varphi(\vartheta, \xi^k)$ satisfies
    \begin{align*}
        \limsup_{\substack{\xi^k, \xi^{k'} \to \xi_0, \xi^k \neq \xi^{k'} \\ l \to l_0}} \frac{\| \nabla g(\xi^{k'})^* l - \nabla g(\xi^k)^* l \|}{\| \xi^{k'} - \xi^k \|} \leq \kappa
    \end{align*}
    and a difference in the second component can be written as
    \begin{multline*}
        g(\xi^k) + \nabla g(\xi^k) (\xi - \xi^k) - g(\xi^{k'}) - \nabla g(\xi^{k'}) (\xi - \xi^{k'}) \\ 
        = \big[ \nabla g(\xi^{k'}) - \nabla g(\xi^k) \big] (\xi - \xi^k) - e(\xi^k, \xi^{k'})
    \end{multline*}
    where $e: (\xi^k, \xi^{k'}) \mapsto g(\xi^{k'}) + \nabla g(\xi^{k'}) (\xi^k - \xi^{k'}) - g(\xi^k)$.
    By definition of the Fréchet derivative and Lipschitz continuity of $\nabla g(\xi)$ around $\xi_0$, we conclude that 
    \begin{align*}
        \limsup_{\xi^k, \xi^{k'} \to \xi_0, \xi^k \neq \xi^{k'}} \frac{e(\xi^k, \xi^{k'})}{\| \xi^{k'} - \xi^k \|} = 0
    \end{align*}
    and
    \begin{align*}
        \limsup_{\substack{\xi^k, \xi^{k'} \to \xi_0, \xi^k \neq \xi^{k'} \\ \xi \to \xi_0}} \frac{\| \nabla g(\xi^{k'}) - \nabla g(\xi^k) \|}{\| \xi^{k'} - \xi^k \|} (\xi - \xi^k) = 0.
    \end{align*}
    Combining these results we obtain
    \begin{align*}
        \limsup_{\substack{\xi^k, \xi^{k'} \to \xi_0, \xi^k \neq \xi^{k'} \\ \vartheta \to \vartheta_0}} \frac{\| \hat \varphi(\vartheta, \xi^{k'}) - \hat \varphi(\vartheta, \xi^k) \|}{\| \xi^{k'} - \xi^k \|}
        \leq \kappa + 0
    \end{align*}
    which is the desired result.
\end{proof}

We continue our theoretical analysis by proving that, by strong regularity of $\hat L(\cdot, \xi^k)$ and Lipschitz continuity of $\hat \varphi(\vartheta_0, \cdot)$, the KKT conditions of \eqref{eq:sos-linearized} have a locally unique solution at $\xi^k$ if $\vartheta^k = (\xi^k, l^k)$ is sufficiently close to $\Theta$.

\begin{proposition}
    \label{prop:localization}
    Let $\vartheta_0 = (\xi_0, l_0) \in \Theta$; the solution map $H$ of \eqref{eq:vi-linearized} %
    has a single-valued localization $\hbar: X \to \mathcal T$ around $\xi_0$ for $\vartheta_0$; and $\hbar$ is Lipschitz continuous around $\xi_0$ with constant~$\gamma \kappa$.
\end{proposition}
\begin{proof}
    Define $\varphi_0: \vartheta \mapsto \hat \varphi(\vartheta, \xi_0)$; then $\varphi_0(\vartheta)$ equals $\hat \varphi(\vartheta, \xi^k)$ around $\vartheta_0$ if $\xi^k \to \xi_0$ by continuity and $\varphi_0 + N = L_0$ is strongly regular with constant $\gamma$ by Assumption~\ref{ass:regularity}. Moreover,
    $\hat \varphi(\vartheta, \xi^k)$ is Lipschitz continuous with respect to $\xi_k$ around $(\vartheta_0, \xi_0)$ with constant $\kappa$ by Lemma~\ref{lem:regularity}. By virtue of Theorem~\ref{thm:implicit-vi}, there exists a Lipschitz continuous, single-valued localization $\hbar$ of $H$ around $\xi_0$ for $\vartheta_0$ with constant $\gamma \kappa$, the desired result.
\end{proof}

In consequence, the convex problems \eqref{eq:sos-linearized} and \eqref{eq:sos-linearized-dual} at $\xi^k$ are not only feasible but have unique solutions around $\xi_0$.

\begin{lemma}
    \label{lem:uniqueness}
    Let $\vartheta_0 = (\xi_0, l_0) \in \Theta$; then $\hbar(\xi^k) = (\xi_+, l_+)$ around $\xi_0$.
\end{lemma}
\begin{proof}
    By Proposition~\ref{prop:localization}, there exists a single-valued localization $\hbar(\xi^k)$ of $H$ around $\xi_0$; that is, $H(\xi^k)$ is nonempty and, by Lemma~\ref{lem:optimality}, a convex set that contains $(\xi_+, l_+)$ around $\xi_0$. On the other hand, $H(\xi^k)$ contains the isolated point $\hbar(\xi^k)$ and thus is a singleton. Hence,
    $\hbar(\xi^k) = (\xi_+, l_+)$ around $\xi_0$.
\end{proof}

It rests to prove that the next iterate, subject to the line search, converges towards a KKT point as well.%

\begin{proposition}
    \label{prop:linesearch}
    Let $(\xi_0, l_0) \in \Theta$; %
    if $\xi^k, \xi_+$, and $l_+$ are sufficiently close to $(\xi_0, l_0)$, then the solution of \eqref{eq:linesearch} satisfies $\hat r \geq \eta/\kappa$. %
\end{proposition}
\begin{proof}
    If $\hat r < 1$, then $\psi'(\hat r) - \eta = 0$.
    The derivative is
    \begin{align*}
        \psi'(r) = \langle f - \nabla g(\xi(r))^* l_+, \xi_+ - \xi^k \rangle
    \end{align*}
    where $\xi(r) =_\text{def} r \xi_+ + (1-r) \xi^k$.
    Since $(\xi_+, l_+) \in H(\xi^k)$,
    \begin{align*}
        \psi'(0) &= \langle f, \xi_+ \rangle - \langle f - \nabla g(\xi^k)^* l_+, \xi^k \rangle - \langle \nabla g(\xi^k)^* l_+, \xi_+ \rangle \\
            &= -\langle \nabla g(\xi^k)^* l_+, \xi_+ - \xi^k \rangle - \langle l_+, g(\xi^k) \rangle - \langle s, \xi^k \rangle \\
            &= - \langle s, \xi^k \rangle - \langle l_+, \varsigma \rangle \leq 0
    \end{align*}
    where the equalities follow from \eqref{eq:sos-linearized} and \eqref{eq:sos-linearized-dual} as well as $\langle f, \xi_+ \rangle = \langle l_+, \gamma_k \rangle$, and the inequality follows from the definition of the dual cone. %
    Since $\nabla g(\xi)^* l$ is Lipschitz continuous around $(\xi_0, l_0)$ by Assumption~\ref{ass:regularity}, we have that
    \begin{align*}
        | \psi'(r) - \psi'(0) | %
            &\leq \kappa \| \xi_+ - \xi^k \|^2 r.
    \end{align*}
    With $\psi'(0) \leq 0$ and $\psi'(\hat r) = \eta > 0$ {as well as $\| \xi_+ - \xi^k \| \leq \| \xi_+ - \xi_0 \| + \| \xi_0 - \xi^k \| < 1$}, if $\hat r < 1$ and $\xi_+$ and $\xi^k$ are sufficiently close to $\xi_0$, we conclude that $\kappa \hat r \geq \eta$.
\end{proof}

We combine our results into a local convergence property of the sequential approach.

\begin{theorem}
    \label{thm:convergence}
    Let $\vartheta_0 = (\xi_0, l_0) \in \Theta$ and suppose that $\gamma \kappa < 1$; there exists a constant $\alpha \in (0, 1)$ such that
    \begin{align}
        \| \vartheta^{k+1} - \vartheta_0 \| \leq \alpha \| \vartheta^k - \vartheta_0 \|
    \end{align}
    if $\vartheta^k \in \mathcal T$ is sufficiently close to $\Theta$.
\end{theorem}
\begin{proof}
    Let $\vartheta^k = (\xi^k, l^k)$; by Proposition~\ref{prop:localization}, there exists a single-valued localization {$\hbar$} of the solution mapping {$H$ around $\xi_0$ for $\vartheta_0$} satisfying
    \begin{align*}
        \| \hbar(\xi^k) - \hbar(\xi^{k'}) \| \leq \gamma \kappa \| \xi^k - \xi^{k'} \|
    \end{align*}
    for any $\xi^k, \xi^{k'} \in X$ in a neighbourhood of $\xi_0$. Then $\vartheta_0 = \hbar(\xi_0)$ as well as $\vartheta_+ = (\xi_+, l_+) = \hbar(\xi^k)$ by {Lemma~\ref{lem:uniqueness}} %
    and %
    \begin{align*}
        \| \vartheta^{k+1} - \vartheta_0 \| &= \| \hat r (\vartheta_+ - \vartheta_0) + (1 - \hat r) (\vartheta^k - \vartheta_0) \| \\
        &\leq \hat r \gamma \kappa \| \xi^k - \xi_0 \| + (1 - \hat r) \| \vartheta^k - \vartheta_0 \| \\
        &\leq (1 - \hat r + \hat r \gamma \kappa) \| \vartheta^k - \vartheta_0 \|
    \end{align*}
    as $\| \xi^k - \xi_0 \| \leq \| \vartheta^k - \vartheta_0 \|$. Since $\hat r \geq \omega > 0$ by Proposition~\ref{prop:linesearch}, setting $\alpha = 1 - \omega (1 - \gamma \kappa) \in [\gamma \kappa, 1)$ is the desired result.
\end{proof}

Our assumptions in the proof of Theorem~\ref{thm:convergence} are similar to the assumptions in \cite{Dinh2010}, which proved local convergence of sequential convex programming in the Euclidean space and without a line search. However, if we omit the line search ($\hat r \equiv 1$) the rate of convergence we obtain in the proof, namely $\alpha = \gamma \kappa$, is better {than} this previous result.

\sisetup{retain-unity-mantissa=false}

\section{Numerical Examples}
\label{sec:examples}
The region of attraction of a nonlinear dynamic system $\dot x = \phi(x)$ is defined as the set of initial conditions for which the system trajectories converge to an equilibrium point, here the origin.
Estimating the region of attraction is a classical problem in nonlinear systems analysis and a recurrent application of sum-of-squares methods, provided that the dynamics are represented by polynomial equations of motion. One aims to find a polynomial Lyapunov candidate function $v \in \mathbb R_d[x]$ that decays along trajectories starting in a sublevel set of $v$, that is, there exists $\iota > 0$ such that
\begin{align}
    \label{eq:roa-dissipativity}
    \dot v(x) = \partial_x v(x) \phi(x) < 0
\end{align}
for all $x \neq 0$ satisfying $v(x) \leq \iota$. If \eqref{eq:roa-dissipativity} is satisfied, then
$\{ x \in \mathbb R^n \, | \, v(x) \leq \iota \}$ 
is an invariant subset of the region of attraction \cite[Theorem~2]{laSalle1960}.

\begin{table}[h]
    \caption{Computation details of sequential sum-of-squares programming for region-of-attraction estimation.}
    \label{tab:roa-details}

    \centering
    
    \begin{tabular}{| l c c c | c c c |}
        \hline
        Dynamics & $d_0$ & $d_1$ & $d_2$ & Final value & Iterations & Time \\
        \hline\hline
        Short-period & 2 & 0 & 2 & \num{-1.515} & 8 & \SI{2.77}{\second} \\ %
            & 4 & 2 & 4 & \num{-1.772} & 9 & \SI{3.70}{\second} \\ %
        \hline
        Longitudinal & 2 & 0 & 4 & \num{-0.354} & 20 & \SI{9.28}{\second} \\ %
            & 4 & 2 & 4 & \num{-2.788} & 16\footnotemark & \SI{25.24}{\second} \\ %
        \hline
    \end{tabular}
\end{table}

\footnotetext{Terminated early due to numerical issues of {\sc Mosek}; the nonlinear sum-of-squares constraint was satisfied with a tolerance of \num{1.52e-5}.}

\pagebreak

In order to lower bound the volume of the region of attraction estimate, a polynomial shape $p \in \mathbb R[x]$ is introduced. For any $\iota > 0$, the nonlinear sum-of-squares problem is given as
\begin{multline}
    \label{eq:roa-optimization}
    \min_{\substack{v \in \mathbb R_{d_0}[x], b \in \mathbb R \\ \varsigma_1 \in \mathbb R_{d_1}[x], \varsigma_2 \in \mathbb R_{d_2}[x]}} \; -b \\
    \text{s.t.} \quad \left\{
    \begin{aligned}
        s_2 (v - \iota) - \partial_x v \, \phi - \varrho &\in \Sigma_{d'}[x] \\
        s_1 (p - b) - v + \iota &\in \Sigma_{d'}[x] \\
        v - \varrho &\in \Sigma_{d_0}[x] \\
        s_1 \in \Sigma_{d_1}[x], \; s_2 &\in \Sigma_{d_2}[x]
    \end{aligned}
    \right.
\end{multline}
where $\varrho$ is small, positive-definite polynomial; we choose $\varrho = \num{1e-6} \sum_{i=1}^n x_i^2$. 
Eq.~\eqref{eq:roa-optimization} has been considered for region of attraction estimation \cite{chakrabortyEtAl2011, Cunis2020c, chakrabortyEtAl2011ii, topcuEtAl2008, tanEtAl2008} as well as, with minor modifications, for reachability \cite{jarvisEtAl2003}, control synthesis \cite{Cunis2020c, jarvisEtAl2003}, or robust stability \cite{topcuEtAl2010, Iannelli2019}. In these works, the bilinearities are split into an iteration of convex and quasiconvex subproblems via coordinate descent. 

We have applied sequential sum-of-squares programming to estimate the region of attraction of the short-period (two states) and longitudinal (four states) dynamics of an airplane.\footnote{See \cite{chakrabortyEtAl2011} and its appendix for details.} The equations of motion $\phi$ are cubic polynomials in the two-state case and quintic polynomials in the four-state case; we fix $\iota \equiv 1$; and solve for quadratic and quartic Lyapunov functions in both cases. 
The degrees of the decision variables $s_1$ and $s_2$ as well as number of iterations, computation time, and final value of the objective are detailed in Tab.~\ref{tab:roa-details}. The initial guess for $v$ has been the quadratic Lyapunov function of the linearized dynamics; $b = 1$; and the variables $s_1$ and $s_2$ have been initialized to homogeneous polynomials given in the appendix. The computations have been terminated once the change in the primal variables was below a absolute tolerance $\epsilon_k \equiv \num{1e-6}$ and the change of the dual variables was below a relative tolerance $\epsilon_k^* = \num{1e-6} \| l^k \|_*$.

\begin{table}[h]
    \caption{Computation details of the iterative approach of \cite{chakrabortyEtAl2011} for region-of-attraction estimation.}
    \label{tab:roa-vsiteration}

    \centering
    
    \begin{tabular}{| l c c c | c c c |}
        \hline
        Dynamics & $d_0$ & $d_1$ & $d_2$ & Final value & Iterations & Time \\
        \hline\hline
        Short-period & 2 & 0 & 2 & \num{-1.514} & 20 & \SI{10.94}{\second} \\
            & 4 & 2 & 4 & \num{-1.760} & 40 & \SI{48.06}{\second} \\
        \hline
        Longitudinal & 2 & 0 & 4 & \num{-0.353} & 22 & \SI{113.21}{\second} \\ %
            & 4 & 2 & 4 & \num{-2.749} & 66 & \SI{465.50}{\second} \\ %
        \hline
    \end{tabular}
\end{table}

For comparison, running the iterative approach of \cite{chakrabortyEtAl2011} took considerably longer without reaching the same final values (Tab.~\ref{tab:roa-vsiteration}); the differences are particularly noticeable for quartic Lyapunov functions.
The computationally most expensive part of both algorithms are the semidefinite relaxations of the convex sum-of-squares subproblems. In the sequential approach, the semi-definite problem is larger (both with respect to the number of matrix variables $N$ and the number of constraints $M$) since all polynomial decision variables are solved for at the same time; yet in the iterative approaches, numerous convex problems are solved in each iteration. We compare the per-iteration effort for both approaches to solve the region of attraction estimation problems. Based on the assumption that the computational cost for a semidefinite problem is roughly of order $N^3 M$ \cite{Peaucelle2002}, Fig.~\ref{fig:complexity} details how the order of the computational cost in each iteration grows with the number of states and degree of polynomials for the region of attraction estimation. Overall, the cost of in each iteration of the sequential approach is about ten times lower than the cost of the iterative approach.

\begin{figure}[h]
    \centering
    
    \newcommand\stack[2]{\ensuremath{\begin{array}{c} #1 \\ #2 \end{array}}}

\begin{tikzpicture}

\begin{semilogyaxis}[%
width=0.8\linewidth,
height=0.2\textwidth,
at={(0,0)},
scale only axis,
ylabel={order of cost},
yminorticks,
ymajorgrids,
yminorgrids,
ytick={1e6,1e9,1e12,1e15},
minor ytick={1e7,1e8,1e10,1e11,1e13,1e14},
xtick={1,2,3,4},
xticklabels={
    \stack{n=2}{d_0 = 2},
    \stack{n=2}{d_0 = 4},
    \stack{n=4}{d_0 = 2},
    \stack{n=4}{d_0 = 4}
},
cycle list name=color,
legend pos=south east, 
legend cell align=left,
]

\addplot
table[x=No,y=cost]{
No  n   d   N       M       iter    cost
1   2   2   47      21      8       2.18E+06
2   2   4   304     69      9       1.94E+09
3   4   2   1405    230     20      6.38E+11
4   4   4   5469    625     16      1.02E+14
};

\addplot
table[x=No,y=cost]{
No  n   d   N       M       iter    cost
1   2   2   41      21      20      9.87E+06
2   2   4   269     69      40      1.30E+10
3   4   2   1407    285     22      1.16E+13
4   4   4   5247    625     66      9.80E+14
};

\legend
{sequential, iterative}

\end{semilogyaxis}
\end{tikzpicture}
    
    \vspace{-1.5em}
    
    \caption{Comparison of computational cost in each iteration of the sequential sum-of-squares and the iterative approach of \cite{chakrabortyEtAl2011}.}
    \label{fig:complexity}
\end{figure}
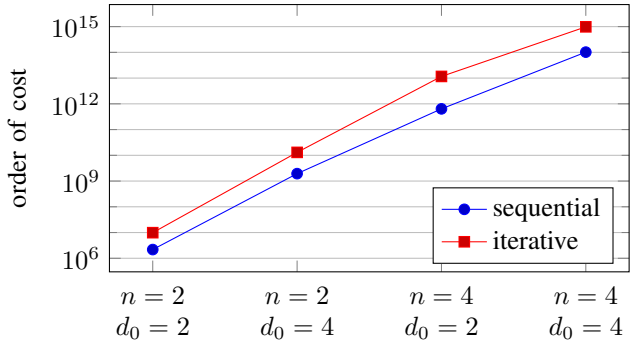

\section{Conclusions}

This paper studies the solution of nonlinear convex programs by a sequential algorithm with the addition of a line search.
Theorem~\ref{thm:convergence} shows that the algorithm still converges with the line search.
As expected, it does not improve the local rate of convergence $\alpha$.
However, we observe in numerical experiments an improvement in the both radius and rate of global convergence and leave proofs of these behaviors as future work.

While the paper analyses the sequential algorithms on arbitrary (possibly infinite dimensional) Banach spaces,
we only apply it to the finite-degree sum-of-squares cone of polynomials with finite degree.
In \cite{legat2022low}, the authors shows how to leverage special properties of the sum-of-squares cone for solving the nonconvex Burer-Monteiro formulation.
We are currently investigating whether such refined analysis could allow the sequential algorithm to exploit the structure of the sum-of-squares cone as well.

\addtolength{\textheight}{-8.7cm}

\section*{Appendix}
For the computations detailed in Tab.~\ref{tab:roa-details}, the following initializations of the multipliers $s_1$ and $s_2$ were used:
\begin{align*}
\begin{alignedat}{4}
    n &= 2, & d_0 &= 2, & s_1 &= 1, & s_2 &= x_1^2 + x_2^2 \\
    n &= 2, & d_0 &= 4, & s_1 &= x_1^2 + x_2^2, \quad & s_2 &= x_1^2 + x_2^2 \\
    n &= 4, & d_0 &= 2, & s_1 &= 1, & s_2 &= {\textstyle \sum_{i=1}^n x_i^2} \\
    n &= 4, \quad & d_0 &= 4, \quad & s_1 &= {\textstyle \sum_{i=1}^n x_i^2}, \quad & s_2 &= ({\textstyle \sum_{i=1}^n x_i^2})^2
\end{alignedat}
\end{align*}
where $n$ is the number of states and $d_0$ the degree of the Lyapunov function.

All numerical examples were performed on a \SI{2.8}{\giga\hertz} quad-core Intel Core i7 processor with \SI{16}{\giga\byte} of memory. We use {\sc sosopt} for the convex sum-of-squares problem and {\sc mosek} for semidefinite programming. Source code is available at \url{https://github.com/tcunis/bisosprob}.

\section*{Acknowledgment}
The authors wish to thank Frank Permenter for initiating their collaboration.

\bibliographystyle{IEEEtran}
\bibliography{main,library}

\end{document}